\documentclass
[11pt,psamsfonts]
{amsart}

\usepackage[left=1.0in,right=1.0in,top=1.0in,bottom=1.0in]{geometry}

\def\T{{\mathbb T}}

\def\Z{{\mathbb Z}}
\def\N{{\mathbb N}}
\def\S{{\mathcal S}}
\def\cont{\mathfrak{c}}

\def\Prm{{\mathbb P}}
\def\snb{strongly unbounded}

\usepackage{amsmath}
\usepackage{amssymb}

\usepackage{mathrsfs}

\newtheorem{thm}{Theorem}[section]
\newtheorem{theorem}[thm]{Theorem}

\newtheorem{corollary}[thm]{Corollary}

\newtheorem{proposition}[thm]{Proposition}
\newtheorem{lemma}[thm]{Lemma}
\newtheorem{fact}[thm]{Fact}

\theoremstyle{definition}
\newtheorem{Def}[thm]{Definition}
\newtheorem{definition}[thm]{Definition}

\newtheorem{remark}[thm]{Remark}

\newtheorem{question}[thm]{Question}

\newtheorem{claim}{Claim}
\newtheorem{problem}[thm]{Problem}

\def\HM{Hartman-Mycielski}
\def\dpc{dp-connected}

\def\HMi{\mathsf{HM}}

\title[A complete solution of Markov's problem]
{A complete solution of Markov's problem on connected group topologies}
\author[D. Dikranjan]
{Dikran Dikranjan}
\address[D. Dikranjan]{\hfill\break
Dipartimento di Matematica e Informatica\\
Universit\`{a} di Udine\\
Via delle Scienze  206, 33100 Udine\\
Italy}
\email{dikran.dikranjan@uniud.it} 

\author[D. Shakhmatov]{Dmitri Shakhmatov}
\address[D. Shakhmatov]{\hfill\break
Division of Mathematics, Physics and Earth Sciences\\
Graduate School of Science and Engineering\\
Ehime University\\
Matsuyama 790-8577\\
Japan}
\email{dmitri.shakhmatov@ehime-u.ac.jp}

\keywords{(locally) connected group, (locally) pathwise connected group, unconditionally closed subgroup, Markov's problem, Hartman-Mycielski construction, Ulm-Kaplanski invariants.}

\subjclass[2010]{Primary: 22A05; Secondary: 20K25, 20K45, 54A25, 54D05, 54H11}

\begin{document}
\begin{abstract}
Every proper closed subgroup of a connected Hausdorff group must have index at least $\cont$, the cardinality of the continuum. 70 years ago Markov conjectured that a group $G$ can be equipped with a connected Hausdorff group topology provided that every subgroup of $G$ which is closed in {\em all\/} Hausdorff group topologies on $G$ has index at least $\cont$.  Counter-examples in the non-abelian case were provided 25 years ago by Pestov and Remus, yet the problem whether Markov's Conjecture holds for abelian groups $G$ remained open. We resolve this problem in the positive.
\end{abstract}
\maketitle

As usual, $\Z$ denotes the group of integers, $\Z(n)$ denotes the cyclic group of order $n$, $\N$ denotes the set of natural numbers,
$\Prm$ denotes the set of all prime numbers, $|X|$ denotes the cardinality of a set $X$,  $\cont$ denotes the cardinality of the continuum and $\omega$ denotes the cardinality of $\N$.

Let $G$ be an abelian group. For a cardinal $\sigma$, we use $G^{(\sigma)}$ to denote the direct sum of $\sigma$ many copies of the group $G$. For $m\in \N$, we let 
$$
mG=\{mg:g\in G\}\  \ \mbox{ and }\  \ G[m]=\{g\in G: mg = 0\},
$$ 
where $0$ is the zero element of $G$.
The group $G$ is {\em bounded\/} (or has {\em finite exponent\/}) if $mG=\{0\}$ for some integer $m\ge 1$;
otherwise, $G$ is said to be {\em unbounded\/}. We denote by 
\begin{equation}
\label{eq:torsion:part}
t(G)=\bigcup_{m\in\N} G[m]
\end{equation}
 the {\em torsion subgroup\/} of $G$. The group $G$ is {\em torsion\/} if $t(G)=G$. 

As usual, we write $G\cong H$ when groups $G$ and $H$ are isomorphic.

{\em All topological groups and all group topologies are assumed to be Hausdorff.\/}

\section{Markov's problem for abelian groups}

Markov \cite{Mar1,Mar} says that a subset $X$ of a group $G$ is {\em unconditionally closed\/} in $G$ if $X$ is closed in every Hausdorff group topology on $G$.

Every proper closed subgroup of a connected group must have index at least $\cont$.\footnote{Indeed, if $H$ is a proper closed subgroup of a connected group $G$, 
the the quotient space $G/H$ is non-trivial, connected and completely regular. Since completely regular spaces of size less than $\cont$ are disconnected, this shows that $|G/H|\ge\cont$.}
Therefore, if a group admits a connected group topology, then all its proper unconditionally closed subgroups necessarily have index at least $\cont$; see \cite{M}. Markov \cite[Problem 5, p. 271]{M} asked if the converse is also true.

\begin{problem}
\label{Markov:question} 
If all proper unconditionally closed subgroups of a group $G$ have index at least $\cont$, does then $G$ admit a connected group topology? 
\end{problem}

\begin{definition}
For 
brevity,
a group satisfying Markov's condition, namely having all proper unconditionally closed subgroups of index at least $\cont$, shall be called an {\em $M$-group\/} (the abbreviation for {\em Markov group\/}).
\end{definition}

Adopting this terminology, Markov's Problem \ref{Markov:question} reads: {\em Does every $M$-group admit a connected group topology?\/}

A negative answer to Problem \ref{Markov:question} was provided first by Pestov \cite{Pe}, then a much simpler counter-example was found by Remus \cite{Re}. Nevertheless, the question remained open in the abelian case, as explicitly stated in \cite{Re} and later in \cite[Question 3.5.3]{CHR} 
as well as in \cite[Question 3G.1 (Question 517)]{Co}.

\begin{question}
\label{abelian:Markov}
Let $G$ be an abelian group all proper unconditionally closed subgroups of which have index at least $\cont$. Does $G$ have a connected group topology?
\end{question} 

An equivalent re-formulation of this question using the notion of an $M$-group reads: {\em Does every abelian $M$-group admit a connected group topology?\/}

Since the notion of an unconditionally closed subset of a group $G$ involves checking closedness of this set in {\em every\/} Hausdorff group topology on $G$, in practice it is hard to decide if a given subgroup of $G$ is unconditionally closed in $G$ or not. In other words, there is no clear procedure that would allow one to check whether a given group is an $M$-group. Our next proposition provides an easy algorithm for verifying whether an abelian group is an $M$-group. Its proof, albeit very short, is based on the fundamental fact that all unconditionally closed subsets of abelian groups are algebraic \cite[Corollary 5.7]{DS-reflection}. 

\begin{proposition}
\label{reformulation:of:Markov:condition}
An abelian group  $G$ is an $M$-group if and only if, for every $m\in\N$, either $mG =\{0\}$ or  $|mG| \ge \cont$. In particular, an abelian group $G$ of infinite exponent  is an $M$-group precisely when $|mG|\ge \cont$ for all integers $m\ge 1$.
\end{proposition}

\begin{proof}
According to \cite[Corollary 5.7]{DS-reflection} and  \cite[Lemma 3.3]{DS_JA},  a proper unconditionally closed subgroup $H$ of $G$ has the form $H = G[m]$ for some 
integer $m\ge 1$.
Since 
$G/G[m]\cong mG$, the index of $H$ coincides with $|G/H| = |G/G[m]|= |mG|$.
\end{proof}

According to the Pr\"ufer theorem \cite[Theorem 17.2]{Fu}, a non-trivial abelian group $G$ of finite exponent is a direct sum of cyclic subgroups 
\begin{equation}
\label{eq:1}
G = \bigoplus _{p \in \pi(G)} \bigoplus_{i=1}^{m_p} \Z(p^i)^{(\alpha_{p,i})},
\end{equation}
where $\pi(G)$ is a non-empty finite set of primes and the cardinals $\alpha_{p,i}$ are known as {\em Ulm-Kaplanski invariants\/} of $G$. Note that while some of them may be equal to zero, the cardinals $\alpha_{p,m_p}$ 
must be positive; they are called {\em leading Ulm-Kaplanski invariants\/} of $G$.  

Based on Proposition \ref{reformulation:of:Markov:condition}, one can re-formulate the Markov property for abelian groups of finite exponent in terms of their Ulm-Kaplanski invariants: 

\begin{proposition}
\label{Kirku*}
A non-trivial abelian group $G$ of finite exponent is an $M$-group if and only if all leading Ulm-Kaplanski invariants of $G$ are at least $\cont$. 
\end{proposition}

\begin{proof} Write the group $G$ as in \eqref{eq:1} and let $k = \prod_{q\in \pi(G)} q^{m_q}$ be the exponent of $G$. In order to compute the leading Ulm-Kaplanski invariant
$\alpha_{p,m_p}$ for $p\in \pi(G)$, let 
$$
k_p = \frac{k}{p}= p^{m_p-1}\cdot \prod_{q\in \pi(G)\setminus \{p\}} q^{m_q}.
$$ 
Then $k_pG \cong \Z(p)^{(\alpha_{p,m_p})}$, so  
\begin{equation}
\label{eq:2}
1< |k_pG|=\begin{cases} p^{\alpha_{p,m_p}}\ \mbox{ if $\alpha_{p,m_p}$
is finite}\\
\alpha_{p,m_p}
\ \mbox{ if $\alpha_{p,m_p}$
is infinite}.
\end{cases}
\end{equation}

If $G$ is an $M$-group, then $|k_pG | \ge \cont$ by \eqref{eq:2} and Proposition \ref{reformulation:of:Markov:condition}, so $\alpha_{p,m_p}\ge \cont$ as well. 

Now suppose that all leading Ulm-Kaplanski invariants of $G$ are at least $\cont$. Fix  $m\in \N$ with $|mG| > 1$. There
exists at least one $p \in \pi(G)$ such that $p^{m_p}$ does not divide $m$.  Let $d$ be the greatest common divisor of $m$ and $k$. Then $mG = dG$, hence from now on we can assume without loss of generality that $m=d$ divides $k$.  As $p^{m_p}$ does not divide $m$, it follows that $m$ divides $k_p$.  Therefore, $k_p G$ is a subgroup of $mG$, and so $|mG| \geq |k_p G| = \alpha_{p,m_p} \geq \cont$ by \eqref{eq:2}. 
This shows that, for every $m\in\N$, either $mG =\{0\}$ or  $|mG| \ge \cont$. Therefore, $G$ is an $M$-group by Proposition \ref{reformulation:of:Markov:condition}.
\end{proof}

Kirku \cite{K} characterized the abelian groups of finite exponent admitting a connected group topology.

\begin{theorem}\label{Kirku}{\cite[Theorem on p. 71]{K}}
If all leading Ulm-Kaplanski invariants of a non-trivial abelian group $G$ of finite exponent are at least $\cont$, then $G$ admits a pathwise connected, locally pathwise connected  group topology.
\end{theorem}

In view of Proposition \ref{Kirku*}, Kirku's theorem can be considered as a partial positive answer to Question \ref{abelian:Markov}.

\begin{corollary}
\label{finite:index:corollary}
Each abelian $M$-group of finite exponent admits a pathwise connected, locally pathwise connected group topology.
\end{corollary} 

\begin{proof}
Obviously, the only group topology on the trivial group is both pathwise connected and locally pathwise connected. For a non-trivial abelian group of finite exponent, the conclusion follows from Proposition \ref{Kirku*} and Theorem \ref{Kirku}.
\end{proof}

The main goal of this paper is to offer a positive solution to Markov's problem in the case complementary to Corollary \ref{finite:index:corollary}, namely, for the class of abelian groups of infinite exponent. In order to formulate our main result, we shall need to introduce an intermediate property between pathwise connectedness and connectedness.

\begin{definition}
We shall say that a space $X$ is {\em densely pathwise connected\/} (abbreviated to 
{\em \dpc\/}) provided that $X$ has a dense pathwise connected subspace.
\end{definition}

Clearly,
\begin{equation}
\label{connected:chain}
\mbox{pathwise connected }\to\mbox{ \dpc\ }\to \mbox{ connected}.
\end{equation}

Obviously, a topological group $G$ is \dpc\ if and only if the arc component of the identity of $G$ is dense in $G$.

As usual, given a topological property $\mathscr{P}$, we say that a space $X$ is {\em locally  $\mathscr{P}$\/} provided that for every point $x\in X$ and each open neighbourhood $U$ of $x$ one can find an open neighbourhood $V$ of $x$ such that the closure of $V$ in $X$ is contained in $U$ and has property $\mathscr{P}$. It follows from this definition and \eqref{connected:chain} that
$$
\mbox{locally pathwise connected }\to\mbox{ locally \dpc\ }\to \mbox{ locally connected}.
$$

\begin{theorem}\label{Main}
Every abelian $M$-group $G$ of infinite exponent admits a \dpc, locally \dpc\  group topology.
\end{theorem}

The proof of this theorem is postponed until Section \ref{proofs}. 
It is worth pointing out that not only the setting of Theorem \ref{Main} is 
complimentary to that of Corollary \ref{finite:index:corollary} but also the proof of Theorem \ref{Main} itself makes no recourse to Corollary \ref{finite:index:corollary}.

The equivalence of items (i) and (ii) in our next corollary offers a  complete solution of  Question \ref{abelian:Markov}.

\begin{corollary}
\label{main:corollary:connected}
 For an abelian group $G$, the following conditions are equivalent:
\begin{itemize}
 \item[(i)] $G$ is an $M$-group;
 \item[(ii)] $G$ admits a connected group topology;
 \item[(iii)] $G$ admits a \dpc, locally \dpc\  group topology.
\end{itemize}
\end{corollary}

A classical theorem of Eilenberg and Pontryagin says that every connected locally compact group has a dense arc component \cite[Theorem 39.4(d)]{St}, so it is \dpc.
In \cite{GM}, one can find examples of pseudocompact connected abelian groups without non-trivial convergent sequences, which obviously implies that their arc components are trivial (and so such groups are very far from being \dpc\ and locally \dpc). Under CH, one can find even a countably compact group with the same property \cite{Tk}. (A rich supply of such groups can also be found in \cite{DS_Forcing}.) Our next corollary is interesting in light of these results.

\begin{corollary}
If an abelian group admits a connected group topology, then it can be equipped with a group topology which is both \dpc\ and locally \dpc.
\end{corollary}

Our last corollary characterizes subgroups $H$ of a given abelian group $G$ 
which can be realized as the connected component of some group topology on $G$.  

\begin{corollary}
For a subgroup $H$ of an abelian group $G$, the following conditions are equivalent:
\begin{itemize}
\item[(i)] $G$  admits a 
group topology $\mathscr{T}$ such that  $H$  coincides with the
connected component of $(G,\mathscr{T})$;
\item[(ii)] $H$  admits a connected 
group topology;
\item[(iii)] $H$  is an  $M$-group.
\end{itemize}
\end{corollary}
\begin{proof}
The implications (i)~$\to$~(ii) and (ii)~$\to$~(iii) are clear.
To check the implication (iii)~$\to$~(i), assume that $H$ is an $M$-group.
Applying the implication (i)~$\to$~(ii) of Corollary 
\ref{main:corollary:connected}, we can find a connected group topology 
$\mathscr{T}_c$ on $H$. 
Let $\mathscr{T}$ be the group topology on $G$ obtained by declaring $(H,\mathscr{T}_c)$ to be an open subgroup of $(G,\mathscr{T})$.
Then $H$ is a clopen connected subgroup of $(G,\mathscr{T})$, which implies that $H$ coincides with the connected component of $(G,\mathscr{T})$.
\end{proof}

The paper is organized as follows.  In Section \ref{Sec:2} we recall
the necessary background on $w$-divisible groups, including a
criterion for dense embeddings into powers of $\T$.  In Section
\ref{Sec:3} we show that uncountable $w$-divisible groups can be
characterized by a property involving large direct sums of unbounded
groups; see Theorem
\ref{characterization:of:strongly:unbounded:groups}.  In Section
\ref{Sec:4} we use this theorem to show that every $w$-divisible group
of size at least $\cont$ contains a $\cont$-homogeneous $w$-divisible
subgroup of the same size; see Lemma
\ref{homogeneous:parts:in:w-divisible:groups}.  (A group $G$ is
$\cont$-homogeneous if $G\cong G^{(\cont)}$.)  The importance of the
notion of $\cont$-homogeneity becomes evident in Section \ref{Sec:5},
where we recall the classical construction, due to Hartman and
Mycielski \cite{HM}, assigning to every group a pathwise connected and
locally pathwise connected topological group. We establish
some additional useful properties of this construction and use them to prove that every $\cont$-homogeneous group admits a pathwise
connected, locally pathwise connected group topology; see Corollary
\ref{c-homogeneous:groups:have:pathwise:connected:group:topologies}.
In Section \ref{proofs}, we combine Lemma
\ref{homogeneous:parts:in:w-divisible:groups}, Corollary
\ref{c-homogeneous:groups:have:pathwise:connected:group:topologies}
and dense embeddings of $w$-divisible groups into powers of $\T$ to
obtain the proof of Theorem \ref{Main}.  Finally, in Section
\ref{Sec:7} we introduce a new cardinal invariant useful in obtaining
a characterization of countable groups that contain infinite direct
sums of unbounded groups, covering the case left open in Theorem
\ref{characterization:of:strongly:unbounded:groups}.

\section{Background on $w$-divisible groups}
\label{Sec:2}

We recall here two fundamental notions from \cite{DGB}.

\begin{definition}
\label{w-divisible:reformulation}
\cite{DGB}
An abelian group $G$ is called {\em $w$-divisible\/} if $|mG|=|G|$ for all integers $m\ge 1$.
\end{definition}

\begin{definition}
\label{def:DGB}
\cite{DGB}
For an abelian group $G$, the cardinal
\begin{equation}
\label{w_d(G)}
w_d(G)=\min\{|nG|:n\in\N\setminus\{0\}\}
\end{equation}
is called the {\em divisible weight\/} of $G$.\footnote{This is a ``discrete case'' of a more general definition given in \cite{DGB} involving the weight of a topological group, which explains the appearance of the word ``weight'' in the term.}
\end{definition}

Obviously, $w_d(G)\le|G|$, and $G$ is $w$-divisible precisely when $w_d(G)=|G|$.  Furthermore, we mention here the following easy fact for future reference:

\begin{fact} \cite[p. 255]{DGB}
If $m$ is a positive integer such that $w_d(G)=|mG|$, then $mG$ is a $w$-divisible subgroup of $G$.
\end{fact}

Using the cardinal $w_d(G)$, the last part of Proposition \ref{reformulation:of:Markov:condition} can be re-stated as follows:
\begin{fact}
\label{w_d(G):restatement:of:M}
An abelian group $G$ of infinite exponent  is an $M$-group if and only if $w_d(G)\ge\cont$.
\end{fact}

\begin{fact}
\label{w_d:and:direct:sums}
\cite[Claim 3.6]{DGB}
Let $n$ be a positive integer, let $G_1,G_2,\dots,G_n$ be abelian groups, and let $G=\bigoplus_{i=1}^n G_i$. 
\begin{itemize}
\item[(i)] $w_d(G)=\max\{w_d(G_i):i=1,2,\dots,n\}$.
\item[(ii)] $G$ is $w$-divisible if and only if there exists index 
$i=1,2,\dots,n$ such that $|G|=|G_i|$ and $G_i$ is $w$-divisible. 
\end{itemize}
\end{fact}

The abundance of $w$-divisible groups is witnessed by the following 

\begin{fact}
\cite[Lemma 4.1]{GM}
\label{basic:decomposition}
Every abelian group $G$ admits a decomposition $G=K\oplus M$, where $K$ is a bounded torsion group and $M$ is a $w$-divisible group.
\end{fact}

An alternative proof of this fact can be found in \cite{DGB1}; a particular case is contained already in \cite{DGB}. 

The next fact describes dense subgroups of $\T^\kappa$ in terms of the invariant $w_d(G)$. (Recall that, for every cardinal $\kappa$,
$\log\kappa=\min\{\sigma:\kappa\le 2^\sigma\}$.)

\begin{fact}
\label{dense:subgroups:in:tori}
For every cardinal $\kappa\ge\cont$,
an abelian group $G$ is isomorphic to a dense subgroup of $\T^{\kappa}$ if and only if $\log\kappa \leq w_d(G) \leq |G|\le 2^{\kappa}$.
\end{fact}
\begin{proof}
For $\kappa=\cont$, this is \cite[Corollary 3.3]{DS_Advances}, and for $\kappa>\cont$ this follows from the equivalence of items (i) and (ii) in \cite[Theorem 2.6]{DS_HMP} and \eqref{w_d(G)}. 
\end{proof}

\begin{corollary}
\label{main:corollary}
Let $\tau$ be an infinite  cardinal. For every abelian group $G$ satisfying $\tau\le w_d(G)\le|G|\le2^{2^\tau}$, there exists a monomorphism $\pi:G\to \T^{2^\tau}$ such that $\pi(G)$ is dense in  $\T^{2^\tau}$.
\end{corollary}
\begin{proof}
Let $\kappa=2^\tau$. Then $\kappa\ge\cont$ and $\log\kappa\le\tau\le w_d(G)\le|G|\le 2^{2^\tau}=2^\kappa$, so the conclusion follows from Fact \ref{dense:subgroups:in:tori}.
\end{proof}

\section{Finding large direct sums in uncountable $w$-divisible groups}
\label{Sec:3}

For every $p\in\Prm$, we use $r_p(G)$ to denote the $p$-rank of an abelian group $G$ \cite[Section 16]{Fu}, while $r_0(G)$ denotes the free rank of $G$ \cite[Section 14]{Fu}.

\begin{Def}
\label{Def:snb} 
Call an abelian group $G$ {\em \snb} if $G$ contains a direct sum $\bigoplus_{i\in I}A_i$ such that $|I|=|G|$ and all groups $A_i$ are unbounded.
\end{Def}

\begin{remark}
\label{snb}
(i) If $H$ is a subgroup of an abelian group $G$ such that $|H|=|G|$ and $H$ is \snb, then $G$ is \snb\ as well.

(ii) An abelian group $G$ satisfying $r_0(G)=|G|$ is \snb.
\end{remark}

\begin{lemma}\label{examples:snb}
A \snb\ abelian group is $w$-divisible.
\end{lemma}

\begin{proof} Let $G$ be a \snb\ abelian group. Then $G$ contains a direct sum $\bigoplus_{i\in I}A_i$ as in Definition \ref{Def:snb}. 

Let $n$ be a positive integer. Clearly, each $nA_i$ is  non-trivial, and $nA=\bigoplus_{i\in I}nA_i$, so $|nG|\ge |nA|\ge |I|=|G|$.
The converse inequality $|nG|\le |G|$ is trivial. Therefore, $G$ is $w$-divisible by Definition \ref{w-divisible:reformulation}.
\end{proof}

In the rest of this section we shall prove the converse of Lemma \ref{examples:snb} for uncountable groups.

\begin{lemma}\label{snb:local:case} Let $p$ be a prime number and $\tau$ be an infinite cardinal. 
For every abelian $p$-group $G$ of cardinality $\tau$, the following conditions are equivalent: 
\begin{itemize}
   \item[(i)] $G$ contains a subgroup algebraically isomorphic to $L_p^{(\tau)}$, where $L_p=\bigoplus_{n\in\N} \Z(p^n)$;
   \item[(ii)] $G$ is \snb;
   \item[(iii)] $r_p(p^nG)\ge\tau$ for every $n\in \N$;
   \item[(iv)] $r_p(S_n)\geq \tau$ for each $n\in \N$, where $S_n=p^nG[p^{n+1}]$.
 \end{itemize}
\end{lemma}

\begin{proof} Since $L_p$ is unbounded and $|G|=\tau$, the implication (i) $\to$ (ii) follows from Definition \ref{Def:snb}.

(ii) $\to$ (iii) Let $n\in\N$ be arbitrary. Since $G$ is \snb, $G$ contains a direct sum $\bigoplus_{i\in I}A_i$  
such that $|I|=|G|=\tau$ and all groups $A_i$ are unbounded.
Thus, the subgroup $p^n G$ of $G$ contains the direct sum $\bigoplus_{i\in I}p^n A_i$ of non-trivial $p$-groups $p^n A_i$. Since $r_p(p^n A_i)\ge 1$ for every $i\in I$, it follows that  $r_p(p^n G)\ge r_p(\bigoplus_{i\in I}p^n A_i)\ge |I|=\tau$.

(iii) $\to$ (iv)
Note that $(p^nG)[p]=S_n$, so $r_p(p^nG)=r_p(S_n)$
for every $n\in\N$.

(iv) $\to$ (i)
Note that each $S_{n+1}$ is a subgroup of the group $S_n$, and therefore, the quotient group $S_n/S_{n+1}$ is well-defined. We 
consider two cases.

\smallskip
{\it Case 1.} {\sl There exists a sequence $0<m_1<m_2<\ldots<m_k<\ldots$ of natural numbers such that $| S_{m_k-1}/S_{m_k}|\geq\tau$ for every $k\in\N$.}  In this  case for every $k\in\N$ one can find a subgroup $V_k$ of $S_{m_k-1}$ of size $\tau$  with $V_k\cap S_{m_k}=\{0\}$. Then the family $\{V_k:k\in\N\}$ is independent; that is, the sum $\sum_{k\in\N} V_k$ is direct.
 Since $V_k$ is a subgroup of $S_{m_k-1}$, one can find a subgroup $N_k$ of $G$ such that $N_k\cong \Z(p^{m_k})^{(\tau)}$ and $N_k[p]=V_k$. Now  the family  $\{N_k:k\in\N\}$ is independent; see \cite[Lemma 3.12]{DS_Forcing}.
The subgroup $N=\bigoplus_{k\in\N} N_k$ of $G$ 
is isomorphic to  $\bigoplus_{k\in\N}\Z(p^{m_k})^{(\tau)}$, so
$N$ contains a copy of  $L_p^{(\tau)}$.

\smallskip
{\it Case 2.} {\sl There exists $m_0\in\N$ such that $|S_m/S_{m+1}|<\tau$ for all $m\geq m_0$.}  Note that in this case $|S_{m_0}/S_{m_0+k}|<\tau$ for every $k\in\N$.  Since $\tau$ is infinite and $r_p(S_{m_0})\ge\tau$ by (iv), we can fix an independent family $\mathscr{V}=\{V_k:k\in\N\}$  in $S_{m_0}$ such that $V_k\cong \Z(p)^{(\tau)}$ for all $k\in\N$. For every $k\in\N$ consider the subgroup $W_k=V_k\cap S_{m_0+k}$ of $V_k$. Since the family $\mathscr{V}$ is independent, so is the family $\mathscr{W}=\{W_k:k\in\N\}$.

Let $k\in\N$. The quotient group $V_k/W_k$ is naturally isomorphic to a subgroup of $S_{m_0}/S_{m_0+k}$,
and thus $|V_k/W_k|\leq |S_{m_0}/S_{m_0+k}|<\tau$. This yields $W_k\cong V_k\cong  \Z(p)^{(\tau)}$.
Since $W_k$ is a subgroup of $S_{m_0+k}$,  there exists a subgroup $N_k$ of $G$ such that  $N_k\cong \Z(p^{m_0+k+1})^{(\tau)}$ and $N_k[p]=W_k$. 

Since the family $\mathscr{W}$ is independent, the family $\mathscr{N}=\{N_k:k\in\N\}$ is independent too; see again \cite[Lemma 3.12]{DS_Forcing}. Therefore, $\mathscr{N}$ generates a subgroup $\bigoplus_{k\in\N}N_k$ of $G$ isomorphic to $\bigoplus_{k\in\N} \Z(p^{m_0+k+1})^{(\tau)}$. It remains only to note that the latter group contains an isomorphic  copy of the group $L_p^{(\tau)}$.
\end{proof}

\begin{lemma}
\label{torsion:w-divisible:are:snb}
An uncountable torsion $w$-divisible group $H$ is \snb.
\end{lemma}

\begin{proof}
Recall that $H=\bigoplus_{p\in\Prm} H_p$, where each $H_p$ is a $p$-group \cite[Theorem 8.4]{Fu}. Let $\tau = |H|$ and $\sigma_p=r_p(H)=r_p(H_p)$ for all $p\in\Prm$. Since $H$ is uncountable, $\tau= \sup_{p\in\Prm} \sigma_p$. We consider two cases.

\smallskip
{\em Case 1\/}. {\sl There exists a finite set $F\subseteq \Prm$ such that $\sup_{p\in\Prm\setminus F} \sigma_p<\tau$.\/} 
In this case, $H=G_1\oplus G_2$, where $G_1=\bigoplus_{p\in F} H_p$ and $G_2= \bigoplus_{p\in\Prm\setminus  F} H_p$. 
Note that $|G_2|\le \omega+\sup_{p\in\Prm\setminus  F} \sigma_p<\tau=|H|$, 
as $H$ is uncountable. Since 
$H=G_1\oplus G_2$
is $w$-divisible, from Fact \ref{w_d:and:direct:sums}~(ii) we conclude that $G_1$ is $w$-divisible and $|G_1|=|H|$. 
This implies that the finite set $F$ is non-empty.
Applying
Fact \ref{w_d:and:direct:sums}~(ii) once again, we can find $p\in F$ such that $H_p$ is $w$-divisible and $|H_p|=|G_1|=|H|=\tau$.  

Let $n\in\N$. Since $H_p$ is a $w$-divisible $p$-group, $|p^n H_p|=|H_p|=\tau$. Since $\tau$ is uncountable and $p^n H_p$ is a $p$-group,
$r_p(p^n H_p)=|H_p|$. Applying the implication  (iii)$\to$(ii) of Lemma \ref{snb:local:case},  we obtain that $H_p$ is \snb. Since $|H_p|=\tau=|H|$, the group $H$ is  \snb\ by Remark \ref{snb}~(i).

\smallskip
{\em Case 2\/}. {\sl $\sup_{p\in\Prm\setminus  F} \sigma_p=\tau$ for all finite sets $ F\subseteq \Prm$.\/} For each $p\in\Prm$ choose a 
$p$-independent
subset $S_p$ of $H_p[p]$ with $|S_p|=\sigma_p$. Let $S=\bigcup_{p\in\Prm} S_p$. One can easily see that the assumption of Case 2  implies the following property:
\begin{equation}
\label{eq:4}
\text{If } Y\subseteq S \text{ and } |Y|<\tau,
\text{ then the set }
\{p\in \Prm: S_p\setminus Y\not=\emptyset\}
\text{ is infinite.}
\end{equation}

By transfinite induction, we shall construct a family  $\{X_\alpha:\alpha<\tau\}$ of pairwise disjoint countable subsets of $S$ such that the set $P_\alpha=\{p\in\Prm: X_\alpha\cap S_p\not=\emptyset\}$ is infinite for each $\alpha<\tau$.
\smallskip
{\em Basis of induction\/}.
By our assumption, infinitely many of sets $S_p$ are non-empty, so we can choose a countable set $X_0\subseteq S$ such that $P_0$ is infinite.

\smallskip
{\em Inductive step\/}.
Let $\alpha$ be an ordinal such that $0<\alpha<\tau$, and suppose that we have already defined  a family  $\{X_\beta:\beta<\alpha\}$ of pairwise disjoint countable subsets of $S$ such that the set $P_\beta$ is infinite for each $\beta<\alpha$. Since $\alpha<\tau$ and $\tau$ is uncountable,  the set $Y=\bigcup\{X_\beta:\beta<\alpha\}$ satisfies $|Y|\le \omega+|\alpha|<\tau$. Using \eqref{eq:4}, we can select a countable subset $X_\alpha$ of $S\setminus Y$ which intersects infinitely many $S_p$'s. 

\smallskip
The inductive construction been complete, for every $\alpha<\tau$, let $A_\alpha$ be the subgroup of $H$ generated by $X_\alpha$. Since $P_\alpha$ is infinite, $A_\alpha$ contains elements of arbitrary large order, so $A_\alpha$ is unbounded.  Since the family $\{X_\alpha:\alpha<\tau\}$ is pairwise disjoint, the sum $\sum_{\alpha<\tau} A_\alpha = \bigoplus_{\alpha<\tau} A_\alpha$ is direct. Since $|H|=\tau$, this shows that $H$ is \snb.
\end{proof}

\begin{thm}  
\label{characterization:of:strongly:unbounded:groups}
An uncountable abelian group is \snb\ if and only if it is $w$-divisible.  
\end{thm}

\begin{proof} The ``only if'' part is proved in Lemma \ref{examples:snb}. To prove the ``if'' part, assume that $G$ is an uncountable $w$-divisible group. 
If $r_0(G)=|G|$, then $G$ is \snb\ by Remark \ref{snb}~(ii). Assume now that $r_0(G)<|G|$. Then $|G|=|H|$ and 
$|G/H|=r_0(G) \cdot \omega < |G|$, where $H = t(G)$ is the torsion subgroup of $G$; see \eqref{eq:torsion:part}.

Let $n$ be a positive integer. Since $nH = H \cap nG$, the quotient group $nG/nH = nG/(H\cap nG)$ is isomorphic to a subgroup of the quotient group $G/H$, so $|nG/nH|\le |G/H|<|G|$. Since $G$ is $w$-divisible, $|G|=|nG|$ by Definition \ref{w-divisible:reformulation}. Now $|nG/nH|<|nG|$ implies $|nH|= |nG|$, as $ |nG| =  |nH| \cdot |nG/nH|$.  Since $|nG|=|G|=|H|$,  this yields $|nH| = |H|$. Since this equation holds for all 
$n\in\N\setminus\{0\}$,
$H$ is $w$-divisible
by Definition \ref{w-divisible:reformulation}.

Since $|H|=|G|$ and $G$ is uncountable by our assumption, 
$H$ is an uncountable torsion group.
Applying Lemma \ref{torsion:w-divisible:are:snb},
we conclude that $H$ is \snb. Since $|H|=|G|$, $G$ is also \snb\ by Remark \ref{snb}~(i).
\end{proof}

Since every unbounded abelian group contains a countable unbounded subgroup, from Theorem \ref{characterization:of:strongly:unbounded:groups} we obtain the following

\begin{corollary}
\label{direct:sums:in:w-divisible:groups}
For every uncountable abelian group $G$, the following conditions are equivalent:
\begin{itemize}
\item[(i)] $G$ is $w$-divisible;
\item[(ii)] $G$ contains a direct sum $\bigoplus_{i\in I} A_i$ of countable unbounded groups $A_i$ such that $|I|=|G|$.
\end{itemize}
\end{corollary}

\begin{remark}\label{Remark:12July}
The assumption that $G$ is uncountable cannot be omitted either in Theorem \ref{characterization:of:strongly:unbounded:groups} or in its Corollary \ref{direct:sums:in:w-divisible:groups}; that is, the converse of Lemma \ref{examples:snb} does not hold for countable groups. Indeed, the group $\Z$ of integer numbers is $w$-divisible, yet it is not \snb. A precise description of 
countable \snb\ groups is given in 
Proposition \ref{Last:proposition} below; see also
Theorem \ref{characterization:snb_groups} that unifies both the countable and uncountable cases.
\end{remark}

\section{$\sigma$-homogeneous groups}
\label{Sec:4}

\begin{definition}
\label{definition:c-homogeneous}
For an infinite cardinal $\sigma$, we say that an abelian group $G$ is {\em $\sigma$-homogeneous\/} if  $G\cong G^{(\sigma)}$.
\end{definition}

The relevance of this definition to the topic of our paper shall become clear from Corollary \ref{c-homogeneous:groups:have:pathwise:connected:group:topologies} below. 

The straightforward proof of the next lemma is omitted.

\begin{lemma}
\label{trivial:group}
\label{single:power:sigma}
\label{combined:power:sigma}
Let $\sigma$ be an infinite cardinal.
\begin{itemize}
\item[(i)] The trivial group is $\sigma$-homogeneous.
\item[(ii)] If 
$\kappa$ is a cardinal with $\kappa\ge\sigma$
and $A$ is an abelian group, then the group $A^{(\kappa)}$ is $\sigma$-homogeneous. 
\item[(iii)] 
If 
$\{G_i:i\in I\}$ is a
non-empty
 family of $\sigma$-homogeneous abelian groups, then
$G=\bigoplus_{i\in I} G_i$ is $\sigma$-homogeneous.
\end{itemize}
\end{lemma}

\begin{lemma}
\label{split:lemma}
Let $\sigma$ be an infinite cardinal. Every abelian bounded torsion group $K$ admits a decomposition $K=L\oplus N$,  where $|L|<\sigma$ and 
$N$ is $\sigma$-homogeneous.
\end{lemma}

\begin{proof}
If $K$ is trivial, then 
$L=N=K$ work, as
$N$
is $\sigma$-homogeneous by Lemma \ref{trivial:group}~(i).
Suppose now that $K$ is non-trivial.
Then
$K$ admits 
a
decomposition  $K=\bigoplus_{i=1}^n K_i^{(\alpha_i)}$, where  
$n\in\N\setminus\{0\}$,
$\alpha_1$, $\alpha_2,\dots,\alpha_n$ are cardinals and each $K_i$ is isomorphic to $\Z(p_i^{k_i})$ for some $p_i\in\Prm$ and $k_i\in\N$. 
Let 
\begin{equation}
\label{IJ}
I=\{i\in \{1,2,\dots,n\}: \alpha_i<\sigma\}
\ \ 
\mbox{and}
\ \ 
J=\{1,2,\dots,n\}\setminus I.
\end{equation}

If $I=\emptyset$, we define $L$ to be the trivial group.
Otherwise, we let $L=\bigoplus_{i\in I} K_i^{(\alpha_i)}$.
Since $K_i$ is finite and $\alpha_i<\sigma$ for every $i\in I$, and $\sigma$ is an infinite cardinal, we conclude that $|L|<\sigma$.

If $J=\emptyset$,  we define $N$ to be the trivial group and
note that $N$ is 
$\sigma$-homogeneous by Lemma \ref{trivial:group}~(i).
If $J\not=\emptyset$,
we let $N=\bigoplus_{j\in J} K_j^{(\alpha_j)}$.
Let 
$i\in J$
be arbitrary.
Since $\alpha_i\ge\sigma$,  Lemma \ref{single:power:sigma}~(ii) implies that $K_i^{(\alpha_i)}$ is  $\sigma$-homogeneous.
Now Lemma \ref{combined:power:sigma}~(iii) implies that $N$ is  $\sigma$-homogeneous  as well.

The equality $K=L\oplus N$
follows from \eqref{IJ} and our definition of $L$ and $N$.
\end{proof}

For a cardinal $\sigma$, we use $\sigma^+$ to denote the smallest cardinal bigger than $\sigma$.

\begin{lemma}
\label{homogeneous:split}
Let $\sigma$ be an infinite cardinal. Every abelian group $G$ satisfying $w_d(G)\ge\sigma$ admits a decomposition $G=N\oplus H$ such that 
$N$ is a bounded $\sigma^+$-homogeneous group and $H$ is a $w$-divisible group with $|H|=w_d(G)$.
\end{lemma}

\begin{proof} Let $G=K\oplus M$ be the decomposition as in Fact \ref{basic:decomposition}. Use \eqref{w_d(G)} to fix an integer $n\ge 1$ such that $w_d(G)=|nG|$. Since $K$ is a bounded group, $mK=\{0\}$ for some positive integer $m$. Now  $nmG=nmK\oplus nmM=nmM$, so $|nmG|=|mnM|=|M|$, as $M$ is $w$-divisible. Since $nmG\subseteq nG$, from the choice of $n$ and \eqref{w_d(G)}, we get $w_d(G)\le |nmG|\le |nG|=w_d(G)$. This shows that $w_d(G)=|M|$.

By Lemma \ref{split:lemma}, there exist a subgroup $L$ of $K$ satisfying $|L|\leq \sigma$ and a   $\sigma^+$-homogeneous subgroup $N$ of $K$ such that $K=L\oplus N$. 
Since $K$ is bounded, so is $N$.
Clearly, 
$G=K\oplus M=L\oplus N\oplus M=N\oplus H$,
where $H=L\oplus M$. 
Since $|M|=w_d(G)\ge\sigma\ge |L|$ and $\sigma$ is infinite,
$|H|=|M|$. In particular, $|H|=w_d(G)$. 
Since $H=L\oplus M$, $|H|=|M|$ and $M$ is $w$-divisible, 
$H$ is $w$-divisible
by Fact \ref{w_d:and:direct:sums}~(ii).
\end{proof}

\begin{lemma}
\label{homogeneous:parts:in:w-divisible:groups}
Every $w$-divisible abelian group $G$ of cardinality at least $\cont$ contains a $\cont$-homo\-geneous $w$-divisible subgroup $H$ such that $|H|=|G|$.
\end{lemma}

\begin{proof}
By Corollary \ref{direct:sums:in:w-divisible:groups}, $G$ contains a subgroup
\begin{equation}
\label{def:A}
A=\bigoplus_{i\in I} A_i,
\end{equation}
 where $|I|=|G|$ and all $A_i$ are countable and unbounded.  In particular, $|A|=|I|=|G|$.  By the same corollary, $A$ itself is $w$-divisible.

\medskip
{\em Case 1.\/} {\sl $|G|>\cont$}.
Note that there are at most $\cont$-many  countable groups, so the sum \eqref{def:A} can be rewritten as 
\begin{equation}
\label{eq:decomp}
A=\bigoplus_{j\in J} A_j^{(\kappa_j)}
\end{equation}
 for a subset  $J$ of $I$ with $|J|\le\cont$ and a suitable family $\{\kappa_j:j\in J\}$ of cardinals such that $|A|=\sup_{j\in J} \kappa_j$.
Define $S=\{j\in J: \kappa_j< \cont\}$. Then \eqref{eq:decomp} becomes
\begin{equation}
\label{A:double:sum}
A=\bigoplus_{j\in J} A_j^{(\kappa_j)}
=
\left(\bigoplus_{j\in S} A_j^{(\kappa_j)}\right)
\oplus
\left(\bigoplus_{j\in J\setminus S} A_j^{(\kappa_j)}\right).
\end{equation}
Since $\kappa_j<\cont$ for all $j\in S$ and $|S|\le|J|\le\cont$,  we have  $\left|\bigoplus_{j\in S} A_j^{(\kappa_j)}\right|\le\cont$.
Since $|A|=|G|>\cont$, equation
\eqref{A:double:sum}
implies that $|G|=|H|$,  where 
\begin{equation}
\label{H:as:a:sum}
H=
\bigoplus_{j\in J\setminus S} A_j^{(\kappa_j)}.
\end{equation}

Let $j\in J\setminus S$ be arbitrary. Since $\kappa_j\ge\cont$ by the choice of $S$, $A_j^{(\kappa_j)}$ is $\cont$-homogeneous by 
Lemma \ref{single:power:sigma}~(ii). From this, \eqref{H:as:a:sum} and Lemma \ref{combined:power:sigma}~(iii), we conclude that  $H$ is $\cont$-homogeneous as well. 

From \eqref{H:as:a:sum}, we conclude that $H$ is a direct sum of $|H|$-many unbounded groups, so $H$ is $w$-divisible by Corollary \ref{direct:sums:in:w-divisible:groups}.

\medskip
{\em Case 2.\/} {\sl $|G|=\cont$}.
For any non-empty  set $ P \subseteq \Prm$ let 
\begin{equation}
\label{eq:Soc}
\mathrm{Soc}_ P(\T) = \bigoplus_{p\in  P}\Z(p).
\end{equation}
For every prime $p\in \Prm$ let $L_p = \bigoplus_{n\in\N}\Z(p^n)$.

It is not hard to realize that every unbounded abelian group contains a subgroup isomorphic to one of the groups from 
the fixed  family  $\S=\S_1\cup\S_2$  of unbounded groups, where
\begin{equation}
\label{eq:S_1:and:S_2}
\S_1 = \{\Z\} \cup \{\Z(p^\infty): p\in \Prm\} \cup \{L_p:p\in \Prm\}
\ 
\mbox{ and }
\ 
\S_2=\{\mathrm{Soc}_ P(\T): P\in [\Prm]^\omega\}.
\end{equation}
(Here $[\Prm]^\omega$ denotes the set of all infinite subsets of $\Prm$.)

It is not restrictive to assume that, in the decomposition \eqref{def:A}, $A_i\in \S$ for each $i\in I$.  For $l=1,2$ define  $I_l=\{i\in I: A_i\in \S_l\}$. We consider two subcases.

\medskip
{\em Subcase A\/}. {\sl $|I_1|=\cont$}.
For every $N\in\S_1$, define $E_N=\{i\in I_1: A_i\cong N\}$. Then $I_1=\bigcup_{N\in\S_1} E_N$.  Since the family $\S_1$ is countable, $|I_1|=\cont$ and $\mathrm{cf}(\cont)>\omega$, 
there exists  $N\in \S_1$ such that $|E_N|=\cont$.  Then $A$ (and thus, $G$) contains the direct sum
\begin{equation}
H=\bigoplus_{i\in E_N} 
N\cong N^{(|E_N|)}=N^{(\cont)}.
\end{equation} 
Clearly, $H$ is $\cont$-homogeneous. Since $|H|=\cont$ and $H$ is a direct sum of $\cont$-many unbounded groups, $H$ is $w$-divisible by Corollary \ref{direct:sums:in:w-divisible:groups}.

\medskip
{\em Subcase B\/}. {\sl $|I_1|<\cont$}.
Since $I=I_1\cup I_2$ and $|I|=\cont$, this implies that $|I_2|=\cont$. By discarding $A_i$ with $i\in I_1$, we may assume, without loss of generality, that  $I=I_2$; that is, $A_i\in \S_2$ for all $i\in I$. From this, \eqref{eq:Soc} and \eqref{eq:S_1:and:S_2}, we conclude that  the direct sum
\eqref{def:A}
can be re-written as
\begin{equation}
\label{A:as:sum:of:Z(p)}
A=\bigoplus_{p\in\Prm} \Z(p)^{(\sigma_p)}
\end{equation}
for a suitable family $\{\sigma_p:p\in\Prm\}$ of cardinals.
\begin{claim}
The set  $C=\{p\in\Prm:\sigma_p=\cont\}$ is infinite.
\end{claim}

\begin{proof} Let $P=\{p_1,p_2,\dots,p_k\}$ be an arbitrary finite set. Define $m=p_1 p_2\dots p_k$. From \eqref{A:as:sum:of:Z(p)},
$$
mA=\bigoplus_{p\in\Prm\setminus P} \Z(p)^{(\sigma_p)},
$$
so 
\begin{equation}
\label{mA:as:countable:sup}
\sup\{\sigma_p:p\in\Prm\setminus P\}=|mA|=|A|=\cont
\end{equation}
as $A$ is $w$-divisible. Note that \eqref{A:as:sum:of:Z(p)} implies that $\sigma_p\le |A|=\cont$ for every $p\in\Prm$. Since  $\mathrm{cf}(\cont)>\omega$ and the set $\Prm\setminus P$ is countable, from \eqref{mA:as:countable:sup} we conclude that $\sigma_p=\cont$ for some  $p\in \Prm\setminus P$. Therefore, $C\setminus P\not=\emptyset$. Since $P$ is an arbitrary finite subset of $\Prm$, this shows that the set  $C=\{p\in\Prm:\sigma_p=\cont\}$ is infinite.
\end{proof}

By this claim, $A$ (and thus, $G$) contains the direct sum
$$
H=\bigoplus _{p\in C} \Z(p)^{(\cont)}\cong 
\left(\bigoplus _{p\in C} \Z(p)\right)^{(\cont)}
=
\mathrm{Soc}_C(\T)^{(\cont)}.
$$

Clearly, $|H|=\cont$. Finally, since $\mathrm{Soc}_C(\T)^{(\cont)}$ is unbounded, $H$ is $w$-divisible by Corollary \ref{direct:sums:in:w-divisible:groups}.
\end{proof}

\section{The \HM\ construction}\label{HM:construction}
\label{Sec:5}

Let $G$ be an abelian group, and let $I$ be the unit interval $[0,1]$. As usual, $G^I$ denotes the set of all functions from $I$ to $G$.
Clearly, $G^I$  is a group under the coordinate-wise operations. For $g\in G$ and $t\in (0,1]$ let $g_t\in G^I$ be the function defined by
$$
g_t(x)=
\left\{
\begin{array}{rl}
g & \mbox{if } x< t \\
0 & \mbox{if } x\ge  t,
    \end{array}
\right.
$$
where $0$ is the zero element of $G$. For each $t\in (0,1]$, $G_t=\{g_t:g\in G\}$ is a subgroup of $G^I$ isomorphic to $G$.
Therefore, $\HMi(G)=\sum_{t\in (0,1]} G_t$ is a subgroup of $G^I$. It is straightforward to check that this sum is direct, so that
\begin{equation}
\label{direct:decomposition:of:HM(G)}
\HMi(G)=\bigoplus_{t\in (0,1]} G_t.
\end{equation}
Since $G_t\cong G$ for each $t\in (0,1]$, from \eqref{direct:decomposition:of:HM(G)} we conclude that 
$\HMi(G)\cong G^{(\cont)}$. It follows from this that $\HMi(G)$ is divisible and abelian whenever $G$ is. Thus, we have proved the 
following:

\begin{lemma}\label{Dima}
For every abelian group $G$, the group $\HMi(G)$ is algebraically isomorphic to $G^{(\cont)}$. 
In particular, if $G$ is divisible and abelian, then so is $\HMi(G)$.
\end{lemma}

In general even a group $G$ with $|G|\ge \cont$ need not be isomorphic to $\HMi(G)$. Indeed, as shown in \cite{HM}, $G$ is contained in $\HMi(G)$
and splits as a semidirect addend. So any group $G$ indecomposable into a non-trivial semi-direct product fails to be algebraically isomorphic to $\HMi(G)$.

When $G$ is an abelian topological group, Hartman and Mycielski \cite{HM} equip $\HMi(G)$ with a topology making it pathwise connected and locally pathwise connected. Let $\mu$ be the standard probability measure on $I$. The {\em \HM\ topology\/} on the group $\HMi(G)$ is the topology generated by taking the family of all sets of the form
\begin{equation}
\label{basic:nghb:in:HM}
O(U,\varepsilon)=\{g\in G^I: \mu(\{t\in I: g(t)\not\in U\})<\varepsilon\},
\end{equation}
where $U$ is an open neighbourhood of $0$ in $G$ and $\varepsilon>0$, as the base at the zero function of $\HMi(G)$.

The next lemma lists two properties of the functor $G\mapsto \HMi(G)$ that are needed for our proofs.

\begin{lemma}\label{HM:groups} 
Let $G$ be an abelian topological group.
\begin{itemize}
\item[(i)] $\HMi(G)$ is pathwise connected and locally pathwise connected.
\item[(ii)] If $D$ is a dense subgroup of $G$, then $\HMi(D)$ is a dense subgroup of $\HMi(G)$.
\end{itemize}
\end{lemma}
\begin{proof}
(i)
For every $g\in G$ and $t\in (0,1]$, the map $f: [0,t]\to \HMi(G)$ defined by $f(s)=g_s$ for $s\in [0,t]$ is a continuous injection, so $f([0,t])$ is a path  in $\HMi(G)$ connecting $f(0)=0$ and $f(t)=g_t$. Combining this with \eqref{direct:decomposition:of:HM(G)}, one concludes that $\HMi(G)$ is pathwise connected. 

To check that $\HMi(G)$ is locally pathwise connected, it suffices to show that every set of the form $O(U,\varepsilon)$ as in 
\eqref{basic:nghb:in:HM} is pathwise connected. Let $U$ be an open neighbourhood of $0$ in $G$ and let $\varepsilon>0$. Fix 
an arbitrary $h\in O(U,\varepsilon)\setminus\{0\}$. For each $s\in [0,1]$,  define
$$
h_s(x)=
\left\{
\begin{array}{rl}
h(x) & \mbox{if } x< s \\
0 & \mbox{if } x\ge  s,
    \end{array}
\right.
$$
and note that $h_s\in O(U,\varepsilon)$. Therefore, the map $\varphi: [0,1]\to \HMi(G)$ defined by $\varphi(s)=h_s$ for $s\in [0,t]$
is continuous, so $\varphi([0,1])$ is a path  in $O(U,\varepsilon)$ connecting $f(0)=0$ and $f(1)=h$.

(ii)
Let $D$ be a dense subgroup of $G$.  First, from the definition of the topology of $\HMi(G)$,  one easily sees that $D_t$ is a dense subgroup of $G_t$ for each $t\in(0,1]$. Let $g\in \HMi(G)$ be arbitrary, and let $W$ be an open subset of $\HMi(G)$  containing $g$. By \eqref{direct:decomposition:of:HM(G)}, $g=\sum_{i=1}^n (g_i)_{t_i}$ for some $n\in \N$, $g_1,g_2,\dots,g_n\in G$ and $t_1,t_2,\dots,t_n\in (0,1]$. Since the group operation in $\HMi(G)$ is continuous, for every $i=1,2,\dots,n$ we can fix an open neighbourhood $V_i$ of $(g_i)_{t_i}$ in $\HMi(G)$ such that $\sum_{i=1}^n V_i\subseteq W$. For each $i=1,2,\dots,n$, we use the fact that $D_{t_i}$ is dense in $G_{t_i}$ to select $d_i\in V_i\cap D_{t_i}$. Clearly, $d=\sum_{i=1}^n d_i\in \sum_{i=1}^n V_i\subseteq W$.
Since $d_i\in D_{t_i}\subseteq \HMi(D)$ for $i=1,2,\dots,n$, 
it follows that $d\in \HMi(D)$. This shows that $d\in \HMi(D)\cap W\not=\emptyset$.
\end{proof}

\begin{corollary}
\label{c-homogeneous:groups:have:pathwise:connected:group:topologies} Every $\cont$-homogeneous group $G$ admits a pathwise connected, locally pathwise connected  group topology.
\end{corollary}

\begin{proof}
Consider $G$ with the discrete topology. By Lemma \ref{HM:groups}~(i), $\HMi(G)$ is pathwise connected and locally pathwise connected. Since $\HMi(G)\cong G^{(\cont)}$ by Lemma \ref{Dima}, and $G^{(\cont)}\cong G$ by our assumption, we can transfer the topology from $\HMi(G)$ to $G$ by means of the isomorphism $\HMi(G)\cong G$. Therefore, $G$ admits a pathwise connected, locally pathwise connected group topology.
\end{proof}

\section{Proof of Theorem \ref{Main}}
\label{proofs}
\label{w-divisible-proof}

\begin{lemma}
\label{embedding:lemma} Let $H$ be a subgroup of an abelian group $G$ and $K$ be a divisible abelian group such that
\begin{equation}
r_p(H)<r_p(K)
\mbox{ and }
r_p(G)\le r_p(K)
\mbox{ for all }
p \in \{0\}\cup \Prm.
\end{equation}
Then every monomorphism $j: H \to K$ can be extended to a monomorphism $j': G \to K$.
\end{lemma}

\begin{proof}
Let us recall some notions used in this proof. Let $G$ be an abelian group. A subgroup $H$ of $G$ is is said to be {\em essential in $G$\/} provided that  $ \langle x \rangle\cap H\not=\{0\}$ for every non-trivial cyclic subgroup $\langle x \rangle$ of $G$. As usual, $\mathrm{Soc}(G)=\bigoplus_{p\in\Prm} G[p]$
denotes the {\em socle\/} of $G$.

By means of Zorn's lemma, we can find a maximal subgroup $M$  of $G$ with the property
\begin{equation}
\label{eq:*}
M \cap H = \{0\}. 
\end{equation}
Then the subgroup $H + M = H \oplus M$ is essential in $G$. Indeed, if $0\ne x \in G$ with $(H + M ) \cap\langle x \rangle
=    \{0\} $, then also $(\langle x \rangle + M ) \cap H =    \{0\}$. Indeed, if 
$$
H \ni h = kx + m \in   \langle x \rangle + M
$$ 
for some $k \in \Z$ and $m\in M$, then 
$$
h - m  = kx \in (H + M ) \cap\langle x \rangle=    \{0\};
$$
 so $kx=0$ and $h=m$, which yields $h=0$ in view of \eqref{eq:*}.
 
Fix a free subgroup $F$ of $M$ such that $M/F$ is torsion. Then $\mathrm{Soc}(M) \oplus F$ is an essential subgroup of $M$.
Let us see first that $j$ can be extended to a monomorphism 
\begin{equation}\label{Ext:lemma}
j_1: H \oplus \mathrm{Soc}(M)  \oplus F \to K. 
\end{equation}
Since $r_0(H) < r_0(K)$ and $r_0(G) \leq r_0(K)$, 
we can find a free subgroup $F'$ of $K$ such that $F' \cap j(H) = \{0\}$ and $r_0(F') = r_0(F)$. On the other hand, for every prime $p\in \Prm$,
$r_p(H) < r_p(K)$ and $r_p(G) \leq r_p(K)$. Hence, we can find in $K[p]$ a subgroup $T_p$ of  $p$-rank $r_p(M)$ with $T_p \cap j(H)[p] =  \{0\}$. Since $F' \cong F$ and $T_p \cong M[p]$ for every prime $p$, we obtain the desired monomorphism $j_1$ as in (\ref{Ext:lemma})  extending $j$.  Since the subgroup $H \oplus \mathrm{Soc}(M)  \oplus F$ of $G$ is essential in $G$, any extension $j': G \to K$ of $j_1$ will be a monomorphism as $j_1$ is a monomorphism. Such an extension $j'$ exists since $K$ is divisible \cite[Theorem 21.1]{Fu}.
\end{proof}

\begin{lemma}
\label{w-divisible} 
Every  $w$-divisible abelian group of size at least $\cont$ admits a \dpc, locally \dpc\ group topology.
\end{lemma}

\begin{proof} Let $G$ be a $w$-divisible abelian group such that $|G|=\tau\ge\cont$. Use Lemma
\ref{homogeneous:parts:in:w-divisible:groups} to find a $w$-divisible subgroup $H$ of $G$ such that $|H|=\tau$ and $H\cong H^{(\cont)}$. Since $H$ is $w$-divisible, $w_d(H)=|H|=\tau$. By Corollary \ref{main:corollary} (applied to $H$), there exists a monomorphism $\pi:H\to \T^{\kappa}$ such that $D=\pi(H)$ is dense in  $\T^{\kappa}$, where $\kappa=2^\tau$. By Lemma \ref{HM:groups}~(ii), $N=\HMi(D)$ is a dense subgroup of 
$K=\HMi(\T^\kappa)$. By Lemma \ref{HM:groups}~(i), $N$ is pathwise connected and locally pathwise connected.

By our choice of $H$, $H\cong H^{(\cont)}$. Since $\pi$ is a monomorphism, $H\cong \pi(H)=D$, so $H^{(\cont)}\cong D^{(\cont)}$.
Finally, $D^{(\cont)}\cong \HMi(D)=N$ by Lemma \ref{Dima}. This allows us to fix a monomorphism $j: H\to K$ such that $j(H)=N$.

Since $\T^\kappa$ is a subgroup of $K$,
$$
|G| = \tau < 2^\tau = \kappa \le r_p(\T^\kappa)\le r_p(K)
\mbox{ for all }
p \in \{0\}\cup \Prm. 
$$
Clearly, $r_p(H)\le |H|\le |G|$ and $r_p(G)\le|G|$ for all $p \in \{0\}\cup \Prm$. Therefore, all assumptions of Lemma \ref{embedding:lemma} are satisfied, and its conclusion allows us to find a monomorphism $j':G\to K$ extending $j$. Since $G'=j'(G)$ contains $j'(H)=j(H)=N$ and $N$ is dense in $K$, $N$ is dense also in $G'$. Since $N$ is pathwise connected and locally pathwise connected,  we conclude that $G'$ is \dpc\ and locally \dpc.

Finally, since $j'$ is an isomorphism between $G$ and $G'$, we can use it to transfer the subspace topology which $G'$ inherits from $K$ 
onto
$G$.
\end{proof}

\medskip 
\noindent {\bf Proof of Theorem \ref{Main}.}
Let $G$ be an $M$-group of infinite exponent. Then $w_d(G)\ge\cont$ by Fact \ref{w_d(G):restatement:of:M}.
Therefore, we can apply Lemma \ref{homogeneous:split} with $\sigma=\cont$ to obtain a decomposition
$G=N\oplus H$, where $N$ is a bounded $\cont^+$-homogeneous group and  $H$ is a $w$-divisible group such that $|H|=w_d(G)\ge\cont$. By Lemma \ref{single:power:sigma}~(ii), $N$ is $\cont$-homogeneous as well. Corollary \ref{c-homogeneous:groups:have:pathwise:connected:group:topologies} guarantees the existence of a pathwise connected, locally pathwise connected group topology $\mathscr{T}_N$ on $N$, and Lemma \ref{w-divisible} guarantees the existence of a \dpc, locally \dpc\ group topology $\mathscr{T}_H$ on $H$. The topology $\mathscr{T}$ of the direct product $(N,\mathscr{T}_N)\times (H, \mathscr{T}_H)$ is \dpc\ and locally \dpc. Since $G=N\oplus H$ and $N\times H$ are isomorphic, $\mathscr{T}$ is the desired topology on $G$.
\qed

\section{Characterization of countable \snb\ groups}
\label{Sec:7}

The description of uncountable \snb\ groups obtained in Theorem \ref{characterization:of:strongly:unbounded:groups} fails for countable groups,
as mentioned in Remark \ref{Remark:12July}. 
For the sake of completeness, we provide here a description of  countable \snb\ groups.

Following \cite{Fu}, for an abelian group $G$, we let 
\begin{equation*}
r(G) = \max\left\{r_0(G), \sum\{r_p(G):p\in \Prm \}\right\}.
\end{equation*}
The
sum in this definition 
may differ
from the supremum $\sup\{r_p(G):p\in \Prm \}$; the latter may be strictly less than the sum, in case all $p$-ranks are finite and uniformly bounded. 
\begin{remark}
\label{remark:about:rank}
\begin{itemize}
\item[(i)] $r(G)>0$ if and only if $G$ is non-trivial.
\item[(ii)] If $r(G)\ge\omega$, then $r(G)=|G|$.
\item[(iii)] If $G$ is uncountable, then $r(G) = |G|$.
\end{itemize}
\end{remark}

Replacing the cardinality $|G|$ with the rank $r(G)$ in Definition \ref{def:DGB}, one gets the following  ``rank analogue'' of the divisible weight:

\begin{definition} 
\label{def:r_d}
For an abelian group $G$, call the cardinal
\begin{equation}
\label{r_d(G)}
r_d(G)=\min\{r(nG):n\in\N\setminus\{0\}\}
\end{equation}
the {\em divisible rank\/} of $G$. We say that $G$ is {\em $r$-divisible\/} if $r_d(G) = r(G)$.
\end{definition}

The notion of divisible rank was defined, under the name of {\em final rank\/}, by Szele \cite{S} for (discrete) $p$-groups. 
The relevance of this notion to the class of \snb\ groups will become clear from Proposition \ref{Last:proposition} below.

Let $G$ be an abelian group. Observe that $r(nG)\le |nG|\le |G|$ for every positive integer $n$. Combining this with \eqref{w_d(G)} and \eqref{r_d(G)}, we get 
\begin{equation}
\label{eq:20}
r_d(G)\le w_d(G)\le |G|.
\end{equation}

\begin{remark}
\label{new:remark}
An abelian group $G$ satisfying $r_d(G)=|G|$ is $r$-divisible.
Indeed, $r_d(G)\le r(G)$ by \eqref{r_d(G)}.
Since $r(G)\le|G|$
holds as well, we conclude that $r_d(G)=r(G)$.
\end{remark}

It turns out that the first inequality in \eqref{eq:20} is strict precisely when $G$ has finite divisible rank.
The algebraic structure of such groups is described in Proposition \ref{groups:of:finite:rank} below.

\begin{lemma}
\label{strict:inequality}
An abelian group $G$ satisfies $r_d(G) < w_d(G)$ if and only if $r_d(G)$ is finite.
\end{lemma}
\begin{proof}
Assume first that $r_d(G)$ is finite.
We shall show that $r_d(G)<w_d(G)$.
If $G$ is bounded, then $nG=\{0\}$ for some $n\in\N\setminus \{0\}$. Therefore, $r_d(G)=0$ by \eqref{r_d(G)}, while $w_d(G)\ge 1$ by \eqref{w_d(G)}. If $G$ is unbounded, then $nG$ must be infinite for every $n\in\N\setminus \{0\}$, and so $w_d(G)\ge\omega$ by \eqref{w_d(G)}. Since $r_d(G)<\omega$, we get  $r_d(G)<w_d(G)$ in this case as well.

Next, suppose that $r_d(G)$ is infinite.
Let $n\in\N\setminus \{0\}$ be arbitrary.
Then $r(nG)\ge r_d(G)\ge\omega$ by \eqref{r_d(G)}.
Therefore, $r(nG)=|nG|$ by Remark \ref{remark:about:rank}~(ii).
Since this equality holds for every $n\in\N\setminus \{0\}$, 
from \eqref{w_d(G)} and \eqref{r_d(G)} we get
$r_d(G)=w_d(G)$.
\end{proof}

The next lemma outlines the most important connections between  the classes of \snb, $r$-divisible and $w$-divisible groups.

\begin{lemma}\label{Last:lemma} 
\begin{itemize}
\item[(i)] If an abelian group $G$ is \snb, then $r_d(G) = w_d(G) = |G|$, so $G$ is $r$-divisible and $r_d(G) \ge \omega$. 
\item[(ii)] Every $r$-divisible group $G$ satisfying $r_d(G) \ge \omega$ is $w$-divisible. 
\item[(iii)] If $G$ is an uncountable $w$-divisible group, then $G$ is $r$-divisible. 
\end{itemize}
\end{lemma}

\begin{proof}  
(i) If $G$ is \snb,  then $G$ contains a direct sum $\bigoplus_{i\in I} A_i$ of unbounded groups $A_i$ such that $|I|=|G|$;
see Definition \ref{Def:snb}. Since for every integer $n > 0$ the group $nA_i$ remains unbounded, $r(nA_i)\ge 1$, so $r(nG)=r(\bigoplus_{i\in I} nA_i)\ge |I|=|G|$. This proves the inequality  $r_d(G) \ge |G|$. 
Combining this with \eqref{eq:20}, we conclude that $r_d(G)=w_d(G)=|G|$.
Thus, $G$ is $r$-divisible by Remark \ref{new:remark}.

(ii) 
Since $G$ is $r$-divisible,
$r(G) = r_d(G)$ by Definition \ref{def:r_d}.
Since $r_d(G)\ge\omega$,
$|G| = r(G)$ by Remark \ref{remark:about:rank}~(ii)
and 
$r_d(G) = w_d(G)$
by Lemma 
\ref{strict:inequality}.
Therefore,
$|G| = w_d(G)$, which implies that  $G$  is $w$-divisible. 

(iii) 
It follows from \eqref{r_d(G)} that $r_d(G)=r(nG)$ for some $n\in\N\setminus\{0\}$.
Since $G$ is $w$-divisible,
 $ |nG|  = |G|$.
Thus, $nG$ 
 is uncountable, and so 
$r(nG)=|nG|$ by Remark \ref{remark:about:rank}~(iii). 
This shows that $r_d(G)=|G|$.
Applying Remark \ref{new:remark}, we conclude that
$G$ is $r$-divisible. 
\end{proof}

Clearly, a \snb\ group must be infinite. Our next proposition describes countable \snb \ groups in terms of their divisible rank $r_d$.

\begin{proposition}\label{Last:proposition} 
A countable abelian group $G$  is \snb \ if and only if $r_d(G)\geq \omega$. 
\end{proposition}

\begin{proof} If $G$ is \snb, then $r_d(G)\ge\omega$ by Lemma \ref{Last:lemma}~(i).

Assume now that $r_d(G)\geq \omega$.  We shall prove that $G$ is \snb.

If $r_0(G)$ is infinite, then 
$r_0(G)=|G|=\omega$, 
so $G$ is \snb\ by Remark \ref{snb}~(ii).

If
$\pi = \{p\in\Prm: r_p(G)>0\}$
is infinite, then $G$ contains a subgroup isomorphic to the \snb \ group $H=\bigoplus_{p\in \pi} \Z(p)$. 
(Indeed, if $\pi=\bigcup\{S_i:i\in\N\}$ is a partition of $\pi$ into pairwise disjoint infinite sets $S_i$, then $H=\bigoplus_{i\in \N} A_i$, where each $A_i=\bigoplus_{p\in S_i} \Z(p)$ is unbounded.)
Since $|G| = |H| = \omega$, Remark \ref{snb}~(i) implies that $G$ is \snb\ as well.

From now on we shall assume that both $r_0(G)$ and $\pi$ are finite. The former assumption implies that $r(mt(G))=\omega$ for every 
integer
$m >0$. The latter assumption entails that  there exists $p \in \pi$ such that $r(mt_p(G)) = r_p(mt_p(G))$ is infinite for all
integers
 $m>0$, where $t_p(G)=\bigcup\{G[p^n]:n\in\N\}$ is the $p$-torsion part of $G$. By the implication (iii)~$\to$~(ii) of Lemma \ref{snb:local:case}, we conclude that 
$t_p(G)$ is \snb. As  $|G| = |t_p(G)| = \omega$, Remark \ref{snb}~(i) yields that $G$ is \snb.
\end{proof}

Finally, we can unify the countable case considered above with Theorem \ref{characterization:of:strongly:unbounded:groups} to obtain a description of {\em all\/} \snb\  groups.

\begin{theorem}\label{characterization:snb_groups}
For an abelian group $G$ the following are equivalent: 
\begin{itemize}
\item[(i)] $G$ is \snb; 
\item[(ii)]  $G$ is $r$-divisible and  $r_d(G)\ge\omega$;
\item[(iii)] $G$ is $w$-divisible and $r_d(G)\ge\omega$.
\end{itemize}
 \end{theorem}

\begin{proof}
The implication (i) $\to$ (ii) follows from Lemma \ref{Last:lemma}~(i). The implication (ii) $\to$ (iii)  follows from Lemma \ref{Last:lemma}~(ii).
 It remains to check the remaining implication (iii) $\to$ (i). For a countable group $G$, it follows from Proposition \ref{Last:proposition}. For an uncountable group $G$, Theorem \ref{characterization:of:strongly:unbounded:groups} applies.  
\end{proof}

Our last
proposition
clarifies
the algebraic structure of abelian groups 
of finite divisible rank.

\begin{proposition}
\label{groups:of:finite:rank}
If $G$ is an abelian group satisfying $r_d(G)<\omega$, then 
$G = G_0 \times D \times B$, where $G_0$ is a torsion-free group, 
$D$ is a divisible torsion group and $B$ is a bounded group such that $r_d(G)=r_0(G_0) + r(D)$.
\end{proposition}

\begin{proof}
Since $r_d(G)<\omega$, it follows from \eqref{r_d(G)} that $r(nG) < \omega$ 
for some integer $n>0$. 
Let $t(nG) = F \times D$, 
where $D$ is a divisible group and $F$ is a reduced group.
Since $r(F)\le r(t(nG))\le r(nG)<\omega$, the reduced torsion group $F$ 
must be finite.
Since $D$ is divisible, it splits as  a direct summand of $nG$ \cite[Theorem 21.2]{Fu};  that is, 
we can find a subgroup $A$ of $nG$ containing $F$ such that $nG = A \times D$. 
Since $t(A) \cong t(nG/D) \cong t(nG)/D \cong F$ is finite, we deduce from \cite[Theorem 27.5]{Fu} that $F$ splits in $A$; 
that is, $A=F\times H$ for some torsion-free (abelian) group $H$. Now  $nG=F \times H \times D$.

Let $m$ be the exponent of the finite group $F$ and let $k=mn$. Then $kG = N \times D $, where $N = mH$ is a torsion-free group and  
$D=mD$ is a divisible torsion group. In particular, $N\cong kG/D$. Since divisible subgroups split,
$G = D \times G_1$ for an appropriate subgroup $G_1$ of $G$.  Multiplying by $k$ and taking into account that $kD =D$, we get  $kG = D \times kG_1$. 
This implies that the group $kG_1\cong kG/D \cong N$ is torsion-free. Hence,  the 
torsion subgroup $t(G_1) = G_1[k]$  of $G_1$ is bounded,
so it splits by 
\cite[Theorem 27.5]{Fu}; 
that is,  $G_1 = t(G_1) \times G_0$, where $G_0$ is a torsion-free group.

We now have obtained the decomposition $G = G_0 \times D \times B$, 
where $B=t(G_1)$ is a bounded torsion group. Recalling \eqref{r_d(G)}, one easily sees that
$r_d(G)=r_0(G_0)+r(D)$. 
\end{proof}

\begin{remark}
Let $\varphi$ be a cardinal invariant of topological abelian groups. In analogy with the divisible weight and the divisible rank, for every topological abelian group $G$, one can define  the cardinal
\begin{equation}
\label{varphi_d(G)}
\varphi_d(G)=\min\{\varphi(nG):n\in\N\setminus\{0\}\}
\end{equation}
and call $G$ {\em $\varphi$-divisible\/} provided that $\varphi_d(G) = \varphi(G)$. (This terminology is motivated by the fact that divisible groups are obviously $\varphi$-divisible.)  We will not pursue here a study of this general cardinal invariant. 
\end{remark}

\section{Open questions}

We finish this paper with a couple of questions. Now that Question \ref{abelian:Markov} is completely resolved, one may try to look at variations of this question.

One such possible variation could be obtained by asking about the existence of a connected group topology on abelian groups having some additional compactness-like properties. Going in this direction, abelian groups which admit a pseudocompact connected group topology were characterized in 
\cite{DS_TP, DS-Memoirs}; the necessary condition was independently established in \cite{CR}. 
Recently, the authors obtained complete characterizations of abelian groups which admit a maximally almost periodic connected group topology, as well as  abelian groups which admit a minimal connected group topology.

Another possible variation of Question \ref{abelian:Markov} is obtained by trying to ensure the existence of a group topology with some stronger connectedness properties. 

\begin{question}
Can every abelian $M$-group be equipped with a pathwise connected (and locally pathwise connected) group topology? 
\end{question}

One may wonder if Theorem \ref{Main} remains valid for groups which are close to being abelian.

\begin{question}
Can Theorem \ref{Main} be extended to all nilpotent groups?
\end{question}

\medskip
\noindent
{\bf Acknowledgments.\/} 
The first named author gratefully acknowledges a FY2013 Long-term visitor grant~L13710 by the Japanese Society for the Promotion of Science (JSPS). The second named author was partially supported by the Grant-in-Aid for Scientific Research~(C) No.~26400091 by the Japan Society for the Promotion of Science (JSPS).

\end{document}